\documentclass[11pt,a4paper]{article}
\usepackage[noadjust]{cite}
\usepackage{mathrsfs}
\usepackage{amsfonts}
\usepackage{amsmath,amsthm}
\usepackage{amsfonts,amssymb,color}
\usepackage{dsfont}
\usepackage{curves}
\usepackage{pifont}
\usepackage{latexsym,amssymb,amsfonts,epsfig,graphicx,cite,psfrag,tikz}
\usepackage{eepic,color,colordvi,amscd}
\usepackage{enumerate}
\usepackage{graphicx}
\usepackage{subfigure}
\usepackage{float}
\usepackage{caption}
\usepackage{authblk}
\usepackage{hyperref}
\usepackage{caption}
\usetikzlibrary{decorations.pathreplacing}

\theoremstyle{plain}
\newtheorem{theorem}{Theorem}[section]

\newtheorem{lemma}[theorem]{Lemma}

\newtheorem{claim}[theorem]{Claim}

\theoremstyle{definition}

\numberwithin{equation}{section}

\newtheorem{remark}[theorem]{Remark}

\renewcommand{\l}{\ell}
\newcommand{\ex}{\textrm{ex}}
\newcommand{\Ex}{\mathrm{Ex}}

\renewcommand{\epsilon}{\varepsilon}

\begin{document}


\title{On the Tur\'an number of double stars}

\author{Ping Hu\thanks{Email: {\tt huping9@mail.sysu.edu.cn.} Supported in part by National Key Research and Development Program of China (2021YFA1002100) and National Natural Science Foundation of China (12471337).} }

\author{Ting Lan\thanks{Email: {\tt lant28@mail2.sysu.edu.cn.}} } 

\affil{School of Mathematics, Sun Yat-sen University, Guangzhou, 510275, China.}

\date{\today}

\maketitle

\begin{abstract}
The Tur\'an number of a graph $F$, $\ex(n,F)$, is the maximum number of edges in a graph on $n$ vertices which does not contain $F$ as a subgraph. Let $S_{a,b}$ denote a double star with a central edge $uv$, $a$ leaves connected to $u$ and $b$ leaves connected to $v$. The function $\ex(n,S_{a,b})$ has been studied for $a=1,2$, their extremal graphs are disjoint copies of $K_{a+b+1}$ and either a small clique or a near $b$-regular graph. In this paper, we further study $\ex(n,S_{3,b})$ and determine the extremal graphs, which have more structures than those of $a=1,2$.

\begin{flushleft}
\end{flushleft}

\end{abstract}

\section{Introduction}
All graphs considered here are finite, simple, and undirected. For any undefined terms in graph theory, see \cite{BM1976}. Let $G$ be a graph and $v\in V(G)$ be a vertex. The \textit{neighborhood} of $v$ is $N_{G}(v)=\{u\in V(G):uv\in E(G)\}$ and $N_{G}[v]=N_{G}(v)\cup \{v\}$. The \textit{degree} of $v$ is $d_{G}(v)=|N_{G}(v)|$. 
For $A,B\subseteq V(G)$, let $e(A,B)$ denote the number of edges between $A$ and $B$. Moreover, we define $d_{A}(v)=|N_{G}(v)\cap A|$. Let $G[A]$ denote the subgraph of $G$ induced by $A$, and let $G-A=G[V(G)\setminus A]$ be the subgraph of $G$ obtained by deleting all vertices in $A$.
Let $P_{t}$, $S_{t}$ $K_{t}$ denote the path, the star and the complete graph on $t$ vertices. Let $K_{s,t}$ denote the complete bipartite graph with parts of size $s$ and $t$. Let $S_{a,b}$ denote the double star with a \textit{central edge} $uv$, $a$ leaves connected to $u$ and $b$ leaves connected to $v$. 
Let $kG$ denote $k$ vertex-disjoint copies of $G$. 
The \textit{distance} between two vertices $u,v\in V(G)$, denoted by $d_{G}(u,v)$ or simply $d(u,v)$, is the minimum number of edges of paths from $u$ to $v$ in $G$.
A \textit{near $r$-regular} graph on $n$ vertices is an $r$-regular graph if $rn$ is even, otherwise it has $n-1$ vertices of degree $r$ and one vertex of degree $r-1$.

For a fixed graph $F$, we say $G$ is \textit{$F$-free} if $G$ does not contain $F$ as a subgraph. The \textit{Tur\'an number} of $F$, denoted by $\ex(n,F)$, is the maximum number of edges in an $F$-free graph on $n$ vertices.
Tur\'an’s classical result \cite{T1941} states that $\ex(n,K_{r+1})=e(T_{r}(n))$ for $n\ge r\ge 2$, where $T_{r}(n)$ is the complete balanced $r$-partite Tur\'an graph on $n$ vertices. Moreover, $T_{r}(n)$ is the unique extremal graph attaining $\ex(n,K_{r+1})$. Erdös and Gallai~\cite{EG1959} proved that $\ex(n,P_{\l})\le (\frac{\l}{2}-1)n$ for $\l\ge 2$ and the equality holds only for the graph with vertex-disjoint copies of $K_{\l-1}$ if $\l-1$ divides $n$. Following this, Faudree and Schelp \cite{FS1975} determined the function $\ex(n,P_{\l})$. Inspired by the result of $P_{\l}$, there is a long-standing conjecture introduced by Erdös and Sós \cite{E1964}: if $T$ is a tree on $t\ge 2$ vertices, then $\ex(n,T)\le (\frac{t}{2}-1)n$. 
For the function $\ex(n,S_{r})$, if $n\le r-1$, then it is easy to see that $\ex(n,S_{r})=e(K_{n})$. If $n\ge r$, then $\ex(n,S_{r})=e(L)$, where $L$ is a near $(r-2)$-regular graph on $n$ vertices, and the extremal graphs are all such graphs $L$. Furthermore, there are many known results for $\ex(n,F)$ when $F$ is a linear forest (e.g. \cite{EG1959, LLP2013, LMY2026, YZ2017a, YZ2021b}) or a star forest (e.g. \cite{LLP2013, LLST2019, LYL2022}) or a path-star forest (e.g. \cite{FCY2026, FY2025, LLP2013, ZW2022}).

Given graphs $H$ and $F$. The \textit{generalized Tur\'an number} $\ex(n,H,F)$ is the maximum number of copies of $H$ in an $F$-free graph on $n$ vertices. Clearly, when $H=K_{2}$ is a single edge, $\ex(n,H,F)$ is the function $\ex(n,F)$. In 2023, Gerbner \cite{G2023} determined the function $\ex(n,K_{k},S_{a,b})$ for $k\ge 3$, as follows.

\begin{theorem}\label{thm:generalized} {\em(\cite{G2023})}
Let $n=p(a+b+1)+q$ with $q\le a+b$. Then, for $k\ge 3$ we have
\[
\ex(n, K_{k}, S_{a,b}) = p \binom{a+b+1}{k} + \binom{q}{k}.
\]
\end{theorem}

In Theorem~\ref{thm:generalized}, an extremal graph is the vertex-disjoint union of $p$ copies of $K_{a+b+1}$ and a copy of $K_{q}$. The $k=2$ case, i.e., the Tur\'an number $\ex(n,S_{a,b})$, is still unknown except when $a=1,2$. In 2011, Sun and Wang \cite{SW2011} determined the exact value of $\ex(n,S_{1,b})$ as follows.

\begin{theorem}\label{thm:a=1} {\em(\cite{SW2011})}
Let $n,b\in \mathbb{N}$ with $n\ge b+3\ge 5$. Let $n=p(b+2)+q$, where $p,q\in \mathbb{N}$ and $0\le q\le b+1$. Then
\[
\ex(n, S_{1,b}) = 
\begin{cases}
  (p-1)\dbinom{b+2}{2} + \left\lfloor \dfrac{b(b+2+q)}{2} \right\rfloor, & 
  \textrm{if } b \ge 4 \textrm{ and } 2 \le q \le b-1; \\[1em]
  p\dbinom{b+2}{2} + \dbinom{q}{2}, &
  \textrm{otherwise}.
\end{cases}
\]
\end{theorem}

In Theorem \ref{thm:a=1}, if $q$ is close to $0$ or $b+1$, then the value of $\ex(n,S_{1,b})$ is attained by the graph $pK_{b+2}\cup K_{q}$. Otherwise, it is attained by $(p-1) K_{b+2}\cup L$, where $L$ is a near $b$-regular graph on $b+2+q$ vertices. Subsequently, Sun, Wang and Wu \cite{SWW2015} studied the function $\ex(n,S_{2,b})$ and obtained a similar result.

\begin{theorem}\label{thm:a=2} {\em(\cite{SWW2015})}
 Let $n,b\in \mathbb{N}$ with $n\ge b+3\ge 4$. Let $n=p(b+3)+q$, where $p,q\in \mathbb{N}$ and $0\le q\le b+2$. Then
\[
\ex(n, S_{2,b}) = 
\begin{cases}
  (p-1)\dbinom{b+3}{2} + \left\lfloor \dfrac{b(b+3+q)}{2} \right\rfloor, & 
\begin{aligned}
  &\textrm{if } b \geq 12 \textrm{ and } 3 \leq q \leq b-2, \textrm{ or if}\\
  &9 \leq b \leq 11 \textrm{ and } 4 \leq q \leq b-3;
\end{aligned} \\[1em]
  p\dbinom{b+3}{2} + \dbinom{q}{2}, & \textrm{otherwise}.
\end{cases}
\]
\end{theorem}

Since $\ex(n,S_{2,b})$ behaves the same as $\ex(n,S_{1,b})$, a natural problem is to determine whether the results for $a=1,2$ generalize to all $a$. In this paper, we determine the exact value of $\ex(n,S_{3,b})$ and characterize its extremal graphs, which shows that such an extension does not hold. Our proof relies on the following lemma that describes the structure of connected $S_{3,b}$-free extremal graphs. For simplicity, we define $\Ex(n, S_{a,b}):=\{G:|V(G)|=n, G\not\supset S_{a,b} \textrm{ and } e(G)=\ex(n,S_{a,b})\}$ as the set of extremal graphs on $n$ vertices.

\begin{lemma}\label{lem:connected}
Let $n,b\in \mathbb{N}$ with $b>3$ and $n\ge b+5$. Let $G\in \Ex(n,S_{3,b})$ and suppose that $G$ is connected, then $n<2(b+4)$ and
\begin{enumerate}[$(i)$]
  \item If $b\le 10$, or if $b\ge 11$ and $n\notin \{2b+1, 2b+2\}$, then $\Delta(G)=b$ and $e(G)=\left\lfloor\frac{bn}{2}\right\rfloor$.
  \item If $b\ge 11$ and $n=2b+1$, then $\Delta(G)=b+1$ and $e(G)=\left\lfloor\frac{bn+3}{2}\right\rfloor$.
  \item If $b\ge 11$ and $n=2b+2$, then $\Delta(G)=b+1$ and $e(G)=\left\lfloor\frac{bn+2+\left\lfloor\frac{b}{2}\right\rfloor}{2}\right\rfloor$.
\end{enumerate} 
\end{lemma}

Lemma~\ref{lem:connected} states that if $n\ge b+5$, then the maximum degree of a connected $S_{3,b}$-free extremal graph is $b$ or $b+1$. Based on Lemma~\ref{lem:connected}, we can prove our main result as follows.

\begin{theorem}\label{thm:main} Let $n,b\in \mathbb{N}$ with $b>3$ and $n\ge b+5$. Let $n=p(b+4)+q$, where $p,q\in \mathbb{N}$ and $0\le q\le b+3$. Then we have
\[
\ex(n, S_{3,b}) = 
\begin{cases}
  (p-1)\dbinom{b+4}{2} + \left\lfloor \dfrac{b(b+4+q)}{2} \right\rfloor, &
  \begin{aligned}
&\textrm{if } b \geq 24 \textrm{ and } 4 \leq q \leq b-4,\textrm{ or if} \\
&16 \leq b \leq 23 \textrm{ and } 5 \leq q \leq b-4,\textrm{ or if} \\
&14 \leq b \leq 15 \textrm{ and } 6 \leq q \leq b-5;\\
\end{aligned}\\
  (p-1)\dbinom{b+4}{2} + \left\lfloor \dfrac{b(2b+1)+3}{2} \right\rfloor, & \textrm{if } b\ge 22 \textrm{ and }q = b-3; \\[1em]
  (p-1)\dbinom{b+4}{2} + \left\lfloor \dfrac{b(2b+2)+2 + \left\lfloor \frac{b}{2} \right\rfloor}{2} \right\rfloor, & \textrm{if }b\ge 34 \textrm{ and } q = b-2; \\[1em]
  p\dbinom{b+4}{2} + \dbinom{q}{2}, & \textrm{otherwise}.
\end{cases}
\] 
\end{theorem}

In Theorem~\ref{thm:main}, if $q$ is close to $0$ or $b+3$, then the value of $\ex(n,S_{3,b})$ is attained by the graph $pK_{b+4}\cup K_{q}$. If $b\ge 22$ and $q=b-3$, then it is attained by the union of $(p-1)K_{b+4}$ and a graph on $2b+1$ vertices with maximum degree $b+1$. If $b\ge 34$ and $q=b-2$, then it is attained by the union of $(p-1)K_{b+4}$ and a graph on $2b+2$ vertices with maximum degree $b+1$. 
Otherwise, it is attained by a near $b$-regular graph on $b+4+q$ vertices. Note that part of the results for $q=b-3$ and $q=b-2$ differ from those in Theorems~\ref{thm:a=1},~\ref{thm:a=2}. This implies that the connected components with more than $b+4$ vertices in extremal graphs are not all near $b$-regular graphs.

The rest of this paper is organized as follows. In Section~\ref{lemmas}, we introduce several lemmas that will be used in the following proofs. And we prove Lemma~\ref{lem:connected} in Section~\ref{sec:con}. Based on Lemma~\ref{lem:connected}, we prove our main result in Section~\ref{sec:main}. In Section~\ref{sec:conclu}, we present some preliminary results and insights on the function $\ex(n,S_{a,b})$ for general $a$.

\section{Preliminaries}\label{lemmas}
In this section, we show some lemmas that will be used in the proofs of Lemma~\ref{lem:connected} and Theorem~\ref{thm:main}. We first introduce the result of $\ex(n, K_{1,b})$, which is related to the function $\ex(n, S_{a,b})$.

\begin{lemma}\label{lem:K1b}{\em(\cite{SW2011})}
Let $n,b\in \mathbb{N}$ with $n\ge b\ge 1$. Then $\ex(n, K_{1,b})=\left\lfloor\frac{(b-1)n}{2}\right\rfloor$.
\end{lemma}

The following lemma is also important for our proofs.

\begin{lemma}\label{lem:n1n2}{\em(\cite{SWW2015})}
Let $n, n_{1}, n_{2}\in \mathbb{N}$ with $n_{1}, n_{2}<n-1$. Then
\[
\binom{n_1}{2} + \binom{n_2}{2} < \min\left\{\binom{n_1 + n_2}{2}, \binom{n-1}{2} + \binom{n_1 + n_2 - n + 1}{2} \right\}.
\]
\end{lemma}

The following lemma establishes bounds on the maximum degree of extremal graphs and characterizes the component containing a vertex that attains the maximum degree.

\begin{lemma}\label{lem:maxdegree}
Let $a,b\in \mathbb{N}$ and $a<b$. Let $n=p(a+b+1)+q$, where $p,q\in \mathbb{N}$ and $0\le q\le a+b$. Suppose $G\in \Ex(n,S_{a,b})$, then we have
\begin{enumerate}[$(i)$]
  \item $b\le \Delta(G)\le a+b$.
  \item If there exists a vertex $u\in V(G)$ such that $d_{G}(u)=a+b$ and suppose $H$ is the component of $G$ containing $u$, then $H=K_{a+b+1}$.
\end{enumerate}
\end{lemma}

\begin{proof}
Note that $K_{1,b+1}$ is a subgraph of $S_{a,b}$, so if $G\in \Ex(n, S_{a,b})$, then we have 
\begin{equation}
e(G)=\ex(n, S_{a,b})\ge \ex(n, K_{1,b+1})=\left\lfloor\frac{bn}{2}\right\rfloor.\label{eq:lowerbound}
\end{equation}
To show $\Delta(G)\ge b$, suppose to the contrary that $\Delta(G)\le b-1$. Then, we obtain $e(G)\le \frac{(b-1)n}{2}< \left\lfloor \frac{bn}{2}\right\rfloor$, which contradicts \eqref{eq:lowerbound}. Hence $\Delta(G)\ge b$ as desired.

If there exists a vertex $u\in V(G)$ with $d_{G}(u)\ge a+b+1$, then all its neighbors are of degree at most $a$. Otherwise, it is easy to see that $G$ must contain $S_{a,b}$. So we can find $a+b+1$ vertices each of degree at most $a$. Let ${G}'$ obtained from $G$ by deleting these $a+b+1$ vertices. Clearly, we delete at most $a(a+b+1)$ edges from $G$. Let ${G}''={G}'\cup K_{a+b+1}$. Since $a<b$, we have
\[
e(G'') \geq e(G)-a(a+b+1)+ \frac{(a+b+1)(a+b)}{2}>e(G).
\]
Note that ${G}''$ is $S_{a,b}$-free and contains more edges than $G$, a contradiction. So we have $\Delta(G)\le a+b$.

Now assume that there is a vertex $u\in V(G)$ with $d_{G}(u)=a+b$, and let $H$ be the component of $G$ containing $u$. Next we prove $H=K_{a+b+1}$. Let $X=\{v_{1},v_{2},...,v_{\l}\}$ be a subset of $N_{G}(u)$ such that for each $v_{i}\in X$, $N_{G}(v_{i})\backslash N_{G}[u]\ne \varnothing$. Note that $d_{G}(v_{i})\le a$ for all $v_{i}\in X$ since $G$ is $S_{a,b}$-free. Let $G^{*}$ be a graph obtained from $G$ by deleting all the edges between $v_{i}~(i=1, 2,..., \l)$ and $N_{G}(v_{i})\backslash u$, and adding all the edges between $v_{i}$ and $N_{G}(u)$. Clearly, we delete at most $\l(a-1)$ edges and add $\frac{\l(a+b-1) + \l(a+b-\l)}{2}$ edges. Then, we have
\[
    e(G^{*}) \geq e(G)-\l(a-1)+\frac{\l(a+b-1)+\l(a+b-\l)}{2}.
\]
Let $ f(\l)=\frac{\l(a+b-1)+\l(a+b-\l)}{2}-\l(a-1) $. If $1 \leq \l \leq a+b$, then we have $f(\l) = \frac{-\l^{2} + \l(2b + 1)}{2} > 0$, so we obtain $e(G^{*})>e(G)$, a contradiction. Therefore, we have $\l=0$, that is $N_{G}(v)\subseteq N_{G}(u)$ for all $v\in N_{G}(u)$. Since $G$ is an extremal graph, we have $H=K_{a+b+1}$. This completes the proof of Lemma~\ref{lem:maxdegree}.
\end{proof}

Lemma~\ref{lem:maxdegree} states that if $G\in \Ex(n,S_{a,b})$, then $\Delta(G)\le a+b$. Furthermore, if $n>a+b+1$ and $G$ is connected, then we have $\Delta (G)<a+b$. Now we are ready to prove Lemma~\ref{lem:connected}.

\section{Proof of Lemma~\ref{lem:connected}}\label{sec:con}
We first give a sketch of the proof. The proof proceeds mainly by contradiction. To characterize the structures of the extremal graphs, we first claim that the extremal graphs possess a certain property. Assume for contradiction that this property does not hold, then we can construct a new graph with more edges, which leads to a contradiction. Using these established properties, we are then able to determine the structures of the extremal graphs. Next, we outline the proof steps as follows.
\begin{itemize}
  \item \textbf{Maximum degree analysis.} Suppose $n<2(b+4)$ and we first prove that $\Delta(G)\le b+1$. To prove this, we assume for contradiction that $\Delta(G)\ge b+2$. Then we consider a vertex of maximum degree and use the $S_{3,b}$-free condition to bound the degrees of its neighbors. This restricts the number of edges in $G$ to be less than that of the new graph we constructed, then we get a contradiction.
  \item \textbf{Case analysis based on $b$ and $n$.} Next, we use the same methods to prove that $\Delta(G)\le b$ except the cases $b\ge 11$ and $n\in \{2b+1,2b+2\}$. Combined with $\Delta(G)\ge b$ from Lemma~\ref{lem:maxdegree}, we know that if $n<2(b+4)$, then (i) of Lemma~\ref{lem:connected} holds. For the cases $b\ge 11$ and $n\in \{2b+1,2b+2\}$, we can construct graphs with maximum degree $b+1$ and more edges than the near $b$-regular graph.
  \item \textbf{Upper bound on $n$.} Finally, we show $n<2(b+4)$ using the conditions of Lemma~\ref{lem:connected}. To prove this, we again proceed by contradiction. Let $n=p(b+4)+q$, where $p,q\in \mathbb{N}$ and $0\le q\le b+3$. We first prove $p\le 4$, and then show that $p\ne 2,3,4$.
\end{itemize}
Next we prove Lemma~\ref{lem:connected}. Since $n\ge b+5$ and $G$ is connected, it follows from Lemma~\ref{lem:maxdegree} that $b\le \Delta(G)\le b+2$.
Suppose $v_{0}\in V(G)$ and $d_{G}(v_{0})=\Delta(G)$. Let $N_{G}(v_{0})=\{v_{1},v_{2},...,v_{\Delta(G)}\}$, $V_{1}=N_{G}[v_{0}]$ and $V_{2}=V(G)\setminus V_{1}$. Let $u_{1},u_{2},...,u_{t}~(t\ge 1)$ be all the vertices such that $d(u_{i},v_{0})=2$, and $v_{1},v_{2},...,v_{s}~(1\le s\le \Delta(G))$ be all the vertices in $N_{G}(v_{0})$ adjacent to some vertices in $\{u_{1},u_{2},...,u_{t}\}$. For $i\in \{1,2,...,s\}$, define $t_{i}=d_{V_{2}}(v_{i})$. Since $G$ is $S_{3,b}$-free, we know that $1\le t_{i}\le 2$ for each $i$. For any subset $S$ of $\{u_{1},u_{2},...,u_{t}\}$, define $V_{S}=V_{1}\cup S$.

\begin{claim}\label{clm:dvi}
If $\Delta(G)=b+1$ or $b+2$, then $d_{G}(v_{i})\le b$ for all $i\in \{1,2,...,s\}$.   
\end{claim} 

\begin{proof}
Assume that there exists $j\in \{1,2,...,s\}$ such that $d_{G}( v_{j})\ge b+1$. Without loss of generality, assume $u_{1}$ is one of the neighbors of $v_{j}$ in $V_{2}$.

\vspace{0.3cm}

\noindent\textbf{Case 1.} $\Delta(G)=b+2$.

Let $S=\{u_{1}\}$, then $V_{S}=V_{1}\cup \{u_{1}\}$ and $|V_{S}|=b+4$. Moreover, since $d_{G}(v_{j})\ge b+1$ and $t_{j}\le 2$, we have $d_{V_{2}\setminus S}(u_{1})\le 3$, as otherwise $G$ contains an $S_{3,b}$ with central edge  $u_1v_j$. Assume there are $\l~(\l\le s\le b+2)$ of $t_{i}$ with $t_{i}=1$ and $s-\l$ of $t_{i}$ with $t_{i}=2$. Note that if $t_{i}=2$ for some $i$, then $d_{V_{1}}(v_{i})=1$, otherwise, $G$ contains an $S_{3,b}$ with central edge $v_iv_0$. Then we have
\[
\begin{cases} 
  e(G[V_{1}]) \leq \dfrac{\l(b+2-1)+(s-\l)\cdot 1 +(b+3-s)(b+2)}{2} = \dfrac{(b+3)(b+2)+\l b-s b-s}{2}; \\ 
  e(V_{1},V_{2}) = \l+2(s-\l)=2s-\l; \\ 
  e(G[V_{2}]) \leq 3+e(G-V_{S}). 
\end{cases}
\]
So we obtain
\[
\begin{aligned}
  e(G) &= e(G[V_{1}])+e(V_{1}, V_{2})+e(G[V_{2}]) \\
  &\leq \frac{(b+3)(b+2)+(3-b)s+(b-2)\l+6}{2} + e(G-V_{S}) \\
  &\leq \frac{(b+3)(b+2)+(b+2)+6}{2}+e\left(G-V_{S}\right) < \frac{(b+4)(b+3)}{2}+e(G-V_{S}) \\
  &= e(K_{b+4} \cup (G-V_{S})),
\end{aligned}
\]
a contradiction. Therefore, we know that if $\Delta(G)=b+2$, then $d_{G}(v_{i})\le b$ for all $i\in \{1,2,...,s\}$.

\vspace{0.3cm}

\noindent\textbf{Case 2.} $\Delta(G)=b+1$.

For any subset $S$ of $\{u_{1},u_{2},...,u_{t}\}$, similar to Case 1, we have
\[
\begin{cases} 
  e(G[V_{1}]) \leq \dfrac{\sum (b+1-t_{i})+(b+2-s)(b+1)}{2}=\dfrac{(b+2)(b+1)-\sum t_i}{2}; \\
  e(V_{1}, V_{2})=\sum t_i; \\
  e(G[V_{2}])=e(G[S])+e(S,V_{2}\setminus S)+e(G-V_{S}).
\end{cases}
\]
and $e(G)=e(G[V_{1}])+e(V_{1}, V_{2})+e(G[V_{2}])$.

\vspace{0.3cm}

\noindent\textbf{Case 2.1} $t_{i}=1$ for all $i\in \{ 1,2,...,s\}$.

In this subcase, note that $t_{1}+\cdots+t_{s}= s\le b+1$. Let $S=\{u_{1}\}$, then we have $V_{S}=V_{1}\cup \{u_{1}\}$ and $|V_{S}|=b+3$. Since $d_{G}(v_{j})\ge b+1$ and $t_{j}=1$, we know $d_{V_{2}\setminus S}(u_{1})\le 2$, as otherwise $G$ contains an $S_{3,b}$ with central edge $u_1v_j$. So $e(G[S])+e(S,V_{2}\setminus S)\le 2$ and then
\[
\begin{aligned}
  e(G)&=e(G[V_{1}])+e(V_{1}, V_{2})+e(G[V_{2}])\\
  &\leq \frac{(b+2)(b+1)+(b+1)+4}{2} + e(G-V_{S}) \\
  &< \frac{(b+3)(b+2)}{2} + e(G-V_{S})\\
  &= e(K_{b+3} \cup (G-V_{S})),
\end{aligned}
\]
a contradiction.

\vspace{0.3cm}

\noindent\textbf{Case 2.2.} There exists $k\in \{ 1,2,...,s\}$ such that $t_{k}=2$.

If $k=j$, without loss of generality, assume $u_{2}$ is another neighbor of $v_{j}$ in $V_{2}$ and let $S=\{u_{1},u_{2}\}$, then we have $V_{S}=V_{1}\cup\{u_{1},u_{2}\}$ and $|V_{S}|=b+4$. Again we know $d_{V_{2}\setminus S}(u_{1})\le 2$, $d_{V_{2}\setminus S}(u_{2})\le 2$. And there may be one edge between $u_{1}$ and $u_{2}$, so we obtain $e(G[V_{2}])\le 5+e(G-V_{S})$.
If $k\ne j$ and $d_{G}(v_{k})\ge b+1$, then the same as case $k=j$, we have $e(G[V_{2}])\le 5+e(G-V_{S})$. If $k\ne j$ and $d_{G}(v_{k})\le b$, without loss of generality, assume $u_{k}\ne u_{1}$ is one of the neighbors of $v_{k}$ and let $S=\{u_{1},u_{k}\}$, then we have $V_{S}=V_{1}\cup \{u_{1},u_{k}\}$ and $|V_{S}|=b+4$. Note that $d_{V_{2}\setminus S}(u_{1})\le 2$, $d_{V_{2}\setminus S}(u_{k})\le b$ and there may be one edge between $u_{1}$ and $u_{k}$. Thus we obtain $e(G[V_{2}])\le b+3+e(G-V_{S})$. According to all the analysis in this subcase and since $t_{1}+\cdots+t_{s}\le 2(b+1)$, we have
\[
\begin{aligned}
  e(G)&=e(G[V_{1}])+e(V_{1}, V_{2})+e(G[V_{2}])\\
  &\leq \frac{(b+2)(b+1)+2(b+1)+2(b+3)}{2} + e(G-V_{S})\\
  &< \frac{(b+4)(b+3)}{2} + e(G-V_{S})\\ 
  &= e(K_{b+4} \cup (G-V_{S})),
\end{aligned}
\]
a contradiction. This completes the proof of Claim \ref{clm:dvi}.
\end{proof}

\begin{claim}\label{clm:nneq}
If $n<2(b+4)$, then for $b\le 10$ or for $b\ge 11$ with $n\ne 2b+1,2b+2$, we have $\Delta(G)=b$ and $e(G)=\left\lfloor\frac{bn}{2}\right\rfloor$.
\end{claim}

\begin{proof} 
Let $n=b+4+q$ with $1\le q\le b+3$. We first consider the case $\Delta(G)=b+2$. Note that $|V_{2}|=n-(b+3)=q+1$, then by Claim~\ref{clm:dvi}, we have
\begin{equation}
\begin{cases}
  e(G[V_{1}]) \leq \dfrac{\sum (b-t_i)+(b+3-s)(b+2)}{2} 
    = \dfrac{(b+3)(b+2)-2s-\sum t_i}{2}; \\
  e(V_{1}, V_{2}) = \sum t_i; \\
  e(G[V_{2}]) \leq \min\left\{\dfrac{(b+2)(q+1)}{2}, \dfrac{q(q+1)}{2}\right\}.\label{eq:b+2}
\end{cases} 
\end{equation}
Since $t_1+\cdots+t_s\leq 2s$, we obtain
\begin{align*}
e(G) &=e(G[V_{1}])+e(V_{1}, V_{2})+e(G[V_{2}])\\
  &\leq \frac{(b+3)(b+2)-2s+\sum t_i}{2} + \min\left\{\frac{(b+2)(q+1)}{2}, \frac{q(q+1)}{2}\right\} \\[6pt]
  &\leq \frac{(b+3)(b+2)}{2} + \min\left\{\frac{(b+2)(q+1)}{2}, \frac{q(q+1)}{2}\right\} \\[6pt]
  &\leq \begin{cases}
  \dfrac{(b+3)(b+2)}{2} + \dfrac{(b+2)(q+1)}{2},~~ & q = b+3, b+2; \\[6pt]
  \dfrac{(b+3)(b+2)}{2} + \dfrac{q(q+1)}{2}, & q \leq b+1.
  \end{cases} \\[6pt]
  &< \frac{(b+4)(b+3)}{2} + \frac{q(q-1)}{2} = e(K_{b+4} \cup K_q),
\end{align*}
a contradiction. So if $n<2(b+4)$, then we have $\Delta(G)\le b+1$. 

Next we consider the case $\Delta(G)=b+1$. Note that $|V_{2}|=n-(b+2)=q+2$.
Let $H_{1}$ be a near $b$-regular graph on $b+4+q$ vertices. Note that $H_{1}$ is $S_{3, b}$-free and $e(H_{1})=\left\lfloor\frac{b(b+4+q)}{2}\right\rfloor$. The same as \eqref{eq:b+2}, we consider three parts of $e(G)$ and then obtain
\begin{equation}
\begin{aligned}
&e(G)=e(G[V_{1}])+e(V_{1}, V_{2})+e(G[V_{2}])\\
  &\leq \frac{(b+2)(b+1)-s+\sum t_i}{2}+\min\left\{\frac{(b+1)(q+2)-\sum t_i}{2},\frac{(q+2)(q+1)}{2}\right\} \\
  &\leq 
  \begin{cases}
  \dfrac{(b+2)(b+1)-s}{2}+\dfrac{(b+1)(q+2)}{2}, ~~& b \leq q \leq b+3;\\[8pt]
  \dfrac{(b+2)(b+1)+s}{2}+\dfrac{(q+2)(q+1)}{2}, & \text{otherwise}.
  \end{cases} \\
  &< 
  \begin{cases}
  \dfrac{(b+4)(b+3)}{2}+\dfrac{q(q-1)}{2}=e(K_{b+4} \cup K_q), ~~& b \leq q \leq b+3 \text{ or } 0\leq q <\frac{3b+7}{4};\\[8pt]
  \left\lfloor\dfrac{b(b+4+q)}{2}\right\rfloor=e(H_1), & b\ge 10 \text{ and }  1 \leq q \leq b-4. \label{eq:main}
  \end{cases}
\end{aligned} 
\end{equation}
The last inequality is due to $1\le s\le b+1$. If $b\le 10$, then $b-1<\frac{3b+7}{4}$, so we have $e(G)<e(K_{b+4} \cup K_q)$ for all $q$, a contradiction. Hence we obtain $\Delta(G)\le b$. If $b\ge 11$ and $q\ne b-3, b-2, b-1$, then we know $e(G)< \max \{e(K_{b+4} \cup K_q), e(H_1)\}$, a contradiction. So we have $\Delta(G)\le b$.


For $b\ge 11$ and $b-3\le q\le b-1$, we claim that $d_{G}(u_{i})\le b$ for any $i\in \{1,...,t\}$. Otherwise, suppose there exists $j\in \{1,...,t\}$ with $d_{G}(u_{j})=b+1=\Delta(G)$. Then we consider $r:=d_{V_{2}}(u_{j})$. Note that $|V_2|=q+2$, so $e(G[V_{2}]) \leq \frac{q(q+1)}{2} + r$. If $5\le r\le q+1$, then we have $d_{G}(v_{i})\le b+4-r$ for any $v_{i}\in N_{G}(u_{j})\cap V_{1}$. Otherwise $|N(v_{i})\backslash N[u_{j}]|\ge 3$ and then $G$ must contain an $S_{3,b}$ with central edge $u_{j}v_{i}$. It means that there are $b+1-r$ vertices in $\{v_{1},v_{2},...,v_{s}\}$ each of degree at most $b+4-r$. Hence we have $\sum_{i=1}^s d_{V_{1}}(v_{i})\le \sum_{i=1}^s(b-t_{i})-(r-4)(b+1-r)$ and then obtain
\[
\begin{cases}
  e(G[V_{1}]) \leq \dfrac{\sum (b-t_{i})-(r-4)(b+1-r)+(b+2-s)(b+1)}{2}; \\
  e(V_{1}, V_{2}) = \sum t_i; \\
  e(G[V_{2}]) \leq \dfrac{q(q+1)}{2} + r.
\end{cases}
\]
Then,
\begin{align*}
  e(G)&=e(G[V_{1}])+e(V_{1}, V_{2})+e(G[V_{2}])\\
  &\leq \frac{(b+2)(b+1)-s+\sum t_i+q(q+1)-(r-4)(b+1-r)+2r}{2} \\[8pt]
  &\leq \frac{(b+2)(b+1)+(b+1)+q(q+1)+(r^{2}-(b+3)r+4(b+1))}{2} \\[8pt]
  &\leq \frac{(b+2)(b+1)+(b+1)+q(q+1)+((q+1)^{2}-(b+3)(q+1)+4(b+1))}{2}.
\end{align*}
The last inequality is due to $r=q+1$ maximizes $r^{2}-(b+3)r$ when $q\ge b-3$ and $5\le r\le q+1$. Together with $q \leq b-1$, we then have
\[
e(G) \leq \frac{(b+7)(b+1)+(2q-b-2)(q+1)}{2} < e(K_{b+4} \cup K_{q}),
\]
which leads to a contradiction. 
If $r\le 4$, then we have 
\[
\begin{cases}
  e(G[V_{1}]) \leq \dfrac{\sum (b-t_{i})+(b+2-s)(b+1)}{2}; \\
  e(V_{1}, V_{2}) = \sum t_i; \\
  e(G[V_{2}]) \leq \dfrac{q(q+1)}{2} + 4.
\end{cases}
\]
And then it is straightforward to check that $e(G)<e(K_{b+4}\cup K_{q})$, a contradiction. So we have $d_{G}(u_{i})\le b$ for any $i\in \{1,...,t\}$. Now we reconsider $e(G)$ in the case $b\ge11$ and $q=b-1$ and show that $\Delta(G)\le b$. Clearly,
\begin{align*}
e(G)=\frac{1}{2} \sum_{v\in V(G)} d_{G}(v)&=\frac{1}{2} \sum_{v\in V_{1}} d_{G}(v)+\frac{1}{2} \sum_{v\in V_{2}} d_{G}(v)\\
  &\leq \frac{bs+(b+2-s)(b+1)+tb+(q+2-t)(q+1)}{2} \\
  &= \frac{(b+2)(b+1)-s+tb+(q+2-t)(q+1)}{2} \\
  &< e(K_{b+4} \cup K_{q}),
\end{align*}
a contradiction. So we obtain $\Delta(G)\le b$ for $b\ge11$ and $q=b-1$.

Since $G\in \Ex(n, S_{3,b})$, by Lemma~\ref{lem:maxdegree}, we have $\Delta(G)\ge b$. Together with the discussions above, we know that if $n<2(b+4)$, then for $b\le 10$ or for $b\ge 11$ with $n\ne 2b+1,2b+2$, we have $\Delta(G)=b$ and $G$ is a near $b$-regular graph with  $\left\lfloor\frac{bn}{2}\right\rfloor$ edges. This completes the proof of Claim~\ref{clm:nneq}. 
\end{proof}

The conditions of Claim~\ref{clm:nneq} are $n<2(b+4)$, $b\le 10$ or $b\ge 11$ with $n\ne 2b+1,2b+2$.
Actually, we can prove $n<2(b+4)$ from the conditions of Lemma~\ref{lem:connected}. This will be placed in Claim~\ref{clm:n<2(b+4)}. And in the next two claims we consider the cases $b\ge 11$ with $n=2b+1$ and $n=2b+2$. Based on the discussions above, we know $\Delta(G)\le b+1$ if $n<2(b+4)$. And if $b\ge 11$ and $2b+1\le n\le 2b+3$, then $d_{G}(u_{i})\le b$ for any $i\in \{1,...,t\}$.

\begin{claim}\label{clm:n=2b+1} 
If $b\ge11$ and $n=2b+1$, then $\Delta(G)=b+1$ and $e(G)=\left\lfloor\frac{bn+3}{2}\right\rfloor$.
\end{claim}

\begin{proof}
Let $n=b+4+q$ with $q=b-3$, we know from \eqref{eq:main} that
\begin{equation}
e(G) \leq \dfrac{(b+2)(b+1)+s}{2} + \dfrac{(b-1)(b-2)}{2}. \label{eq:q=b-3}
\end{equation}

We first construct a graph $H_{2}$ that the equality in \eqref{eq:q=b-3} holds and we adopt the same notation as for $G$ to describe the corresponding objects for all constructed graphs in this paper.
As shown in Fig.~\ref{fig:H2} (for simplicity, we omit the internal structures of the two ellipses in all figures.):

\begin{itemize}
  \item $d_{V_{2}}(v_{i})=2$ for each $i\in \{1,2,...,b-1\}$;
  \item $d_{H_{2}}(v_{0}) = d_{H_{2}}(v_{b}) = d_{H_{2}}(v_{b+1}) = b+1$;
  \item $H_{2}[V_{2}]\cong K_{b-1}$;
  \item $H_{2}[\{v_{1},...,v_{b-1}\}]$ is a $(b-5)$-regular graph if $b$ is odd, otherwise it has $b-2$ vertices of degree $b-5$ and one vertex of degree $b-6$.
\end{itemize}
Note that $H_{2}$ is $S_{3,b}$-free with $e(H_{2})=\left\lfloor\frac{b(2b+1)+3}{2}\right\rfloor=\frac{(b+2)(b+1)+b-1}{2} + \frac{(b-1)(b-2)}{2}$.

\hfill

\begin{figure}[H]
\centering
\begin{minipage}[t]{0.48\textwidth}
  \centering
  \resizebox{\linewidth}{!}{%
    \begin{tikzpicture}[
        node/.style={circle, draw, fill=black, inner sep=1.5pt}, 
        line width=0.8pt,
        font=\small
    ]

        \node[node] (v0) at (0, 4.0) {};
        \node at (0, 4.3) {$v_0$}; 
        
        \node[node] (v1) at (-3.5, 2.2) {};
        \node[node] (v2) at (-2.5, 2.2) {};
        \node[node] (v3) at (-1.5, 2.2) {};
        \node[node] (v4) at (-0.5, 2.2) {};
        \node at (0,2.2) {$\dots$};
        \node[node] (v5) at (0.5, 2.2) {};
        \node[node] (v6) at (1.5, 2.2) {};
        \node[node,label=below:$v_b$] (vb) at (2.5, 2.2) {};
        \node[node,label=below:$v_{b+1}$] (vb+1) at (3.5, 2.2) {};
        
        \draw[blue, thick] (-3.8, 1.9) rectangle (1.8, 2.5); 
        \node at (-3.4, 3.4) {$\{v_i|i=1,2,\dots,b-1\}$};
        \draw[blue] (-3.7, 3.0) -- (-3.4, 2.4); 

        \draw[lightgray] (0, 2.2) ellipse (4.6cm and 0.9cm); 
        \draw[lightgray] (0, 0) ellipse (4.6cm and 0.9cm); 
        \node at (-5.3, 2.2) {$V_{1}\setminus v_{0}$}; 
        \node at (-5.2, 0) {$V_{2}$}; 
        
        \node[node,label=below: $u_1$] (u1) at (-3.5, 0) {};
        \node[node,label=below: $u_2$] (u2) at (-2.5, 0) {};
        \node[node,label=below: $u_3$] (u3) at (-1.5, 0) {};
        \node[node,label=below: $u_4$] (u4) at (-0.5, 0) {};
        \node at (0,0) {$\dots$};
        \node[node,label=below: $u_{b-2}$] (u5) at (0.5, 0) {};
        \node[node,label=below: $u_{b-1}$] (u6) at (1.5, 0) {};
        
        \draw (v0) -- (v1);
        \draw (v0) -- (v2);
        \draw (v0) -- (v3);
        \draw (v0) -- (v4);
        \draw (v0) -- (v5);
        \draw (v0) -- (v6);
        \draw (v0) -- (vb);
        \draw (v0) -- (vb+1);

        \draw (v1) -- (u1); 
        \draw (v1) -- (u2);
        \draw (v2) -- (u1);
        \draw (v2) -- (u2);
        \draw (v3) -- (u3);
        \draw (v3) -- (u4);
        \draw (v4) -- (u3);
        \draw (v4) -- (u4);
        \draw (v5) -- (u5); 
        \draw (v5) -- (u6); 
        \draw (v6) -- (u5); 
        \draw (v6) -- (u6); 
\end{tikzpicture}}
\vspace{5pt} 
\footnotesize{$b$ is odd} 
\end{minipage}
\hfill
\begin{minipage}[t]{0.48\textwidth}
\centering
\resizebox{\linewidth}{!}{%
\begin{tikzpicture}[
        node/.style={circle, draw, fill=black, inner sep=1.5pt}, 
        line width=0.8pt,
        font=\small
    ]

        \node[node] (v0) at (0, 4.0) {};
        \node at (0, 4.3) {$v_0$}; 
        
        \node[node] (v1) at (-3.7, 2.2) {};
        \node[node] (v2) at (-2.8, 2.2) {};
        \node at (-2.3,2.2) {$\dots$};
        \node[node] (v3) at (-1.7, 2.2) {};
        \node[node] (v4) at (-0.8, 2.2) {};
        \node[node] (v5) at (0.1, 2.2) {};
        \node[node] (v6) at (1.0, 2.2) {};
        \node[node] (v7) at (1.9, 2.2) {};
        \node[node,label=below:$v_b$] (vb) at (2.9, 2.2) {};
        \node[node,label=below:$v_{b+1}$] (vb+1) at (3.7, 2.2) {};
        
        \draw[blue, thick] (-3.9, 1.9) rectangle (2.1, 2.5); 
        \node at (-3.4, 3.4) {$\{v_i|i=1,2,\dots,b-1\}$};
        \draw[blue] (-3.8, 3.0) -- (-3.5, 2.4); 

        \draw[lightgray] (0, 2.2) ellipse (4.6cm and 0.9cm); 
        \draw[lightgray] (0, 0) ellipse (4.6cm and 0.9cm); 
        \node at (-5.3, 2.2) {$V_{1}\setminus v_{0}$}; 
        \node at (-5.2, 0) {$V_{2}$}; 
        
        \node[node,label=below: $u_1$] (u1) at (-3.7, 0) {};
        \node[node,label=below: $u_2$] (u2) at (-2.8, 0) {};
        \node at (-2.3,0) {$\dots$};
        \node[node,label=below: $u_{b-5}$] (u3) at (-1.7, 0) {};
        \node[node,label=below: $u_{b-4}$] (u4) at (-0.8, 0) {};
        \node[node,label=below: $u_{b-3}$] (u5) at (0.1, 0) {};
        \node[node,label=below: $u_{b-2}$] (u6) at (1.0, 0) {};
        \node[node,label=below: $u_{b-1}$] (u7) at (1.9, 0) {};

        \draw (v0) -- (v1);
        \draw (v0) -- (v2);
        \draw (v0) -- (v3);
        \draw (v0) -- (v4);
        \draw (v0) -- (v5);
        \draw (v0) -- (v6);
        \draw (v0) -- (v7);
        \draw (v0) -- (vb);
        \draw (v0) -- (vb+1);
        
        \draw (v1) -- (u1); 
        \draw (v1) -- (u2);
        \draw (v2) -- (u1);
        \draw (v2) -- (u2);
        \draw (v3) -- (u3);
        \draw (v3) -- (u4);
        \draw (v4) -- (u3);
        \draw (v4) -- (u4);
        \draw (v5) -- (u5); 
        \draw (v5) -- (u6); 
        \draw (v6) -- (u5); 
        \draw (v6) -- (u7); 
        \draw (v7) -- (u6);
        \draw (v7) -- (u7);       
\end{tikzpicture}}
\vspace{5pt}
\footnotesize{$b$ is even} 
\end{minipage}
\caption{Two cases of $H_2$}
\label{fig:H2}
\end{figure}
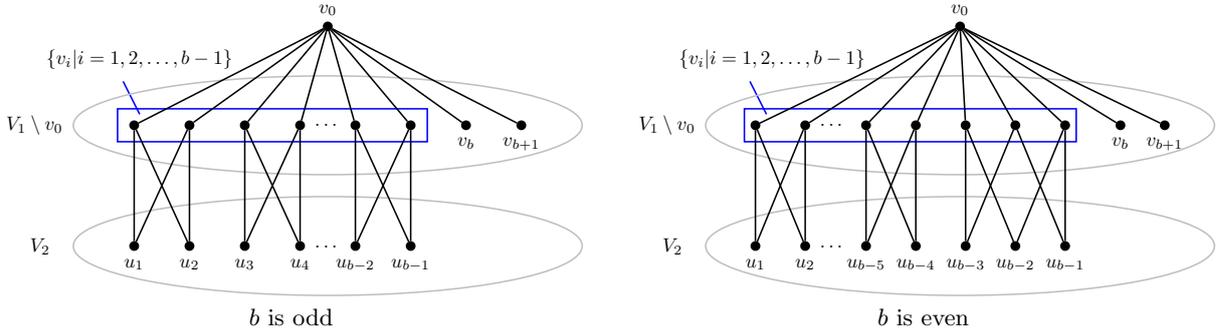

\hfill

If $s\le b-1$, then by \eqref{eq:q=b-3}, we have
\[
e(G) \leq \frac{(b+2)(b+1)+b-2}{2} + \frac{(b-1)(b-2)}{2} \leq \left\lfloor \frac{b(2b+1)+3}{2} \right\rfloor = e(H_{2}).
\]
If $s\ge b$, then by Claim~\ref{clm:dvi}, there are at most two vertices in $V_{1}$ with degree $b+1$ and all other vertices in $V_1$ have degree at most $b$. 
For vertices in $V_2$, we know $d_{G}(u_{i})\le b$ for any $i\in \{1,...,t\}$, and other vertices in $V_2$ have degree at most $|V_2|-1<b$.
So we still have $e(G)\le e(H_{2})$. Hence we have $\Delta(G)=b+1$ and $e(G)=e(H_{2})=\left\lfloor\frac{b(2b+1)+3}{2}\right\rfloor$. This completes the proof of Claim~\ref{clm:n=2b+1}. 
\end{proof}

\begin{claim}\label{clm:n=2b+2}
If $b\ge11$ and $n=2b+2$, then $\Delta(G)=b+1$ and $e(G)=\left\lfloor\frac{bn+2+\left\lfloor\frac{b}{2}\right\rfloor}{2}\right\rfloor$.
\end{claim}

\begin{proof}
Let $n=b+4+q$ with $q=b-2$, we know from \eqref{eq:main} that
\begin{equation}
e(G) \leq \frac{(b+2)(b+1)+s}{2} + \frac{b(b-1)}{2}. \label{eq:q=b-2}
\end{equation}
Similar to the proof of Claim~\ref{clm:n=2b+1}, we first give a construction $H_{3}$ that the equality in \eqref{eq:q=b-2} holds.
As shown in Fig.~\ref{fig:H3}:

\begin{itemize}
  \item $d_{V_{2}}(v_{i})=2$ for each $i\in \{1,2,...,\left\lfloor\frac{b}{2}\right\rfloor\}$;
  \item $d_{H_{3}}(v_{0}) = d_{H_{3}}(v_{\left\lfloor\frac{b}{2}\right\rfloor+1}) =\cdots= d_{H_{3}}(v_{b+1}) = b+1$;
  \item $H_{3}[V_{2}]\cong K_{b}$;
  \item $H_{3}[\{v_{1},...,v_{\left\lfloor\frac{b}{2}\right\rfloor}\}]$ is a $(\left\lfloor\frac{b}{2}\right\rfloor-4)$-regular graph if $\left\lfloor\frac{b}{2}\right\rfloor$ is even, otherwise it has $\left\lfloor\frac{b}{2}\right\rfloor-1$ vertices of degree $\left\lfloor\frac{b}{2}\right\rfloor-4$ and one vertex of degree $\left\lfloor\frac{b}{2}\right\rfloor-5$.
\end{itemize}
It is easy to see that $H_{3}$ is $S_{3,b}$-free with $e(H_{3})=\left\lfloor\frac{b(2b+2)+2+\left\lfloor\frac{b}{2}\right\rfloor}{2}\right\rfloor=\frac{(b+2)(b+1)+\left\lfloor\frac{b}{2}\right\rfloor}{2} + \frac{b(b-1)}{2}$.

\hfill

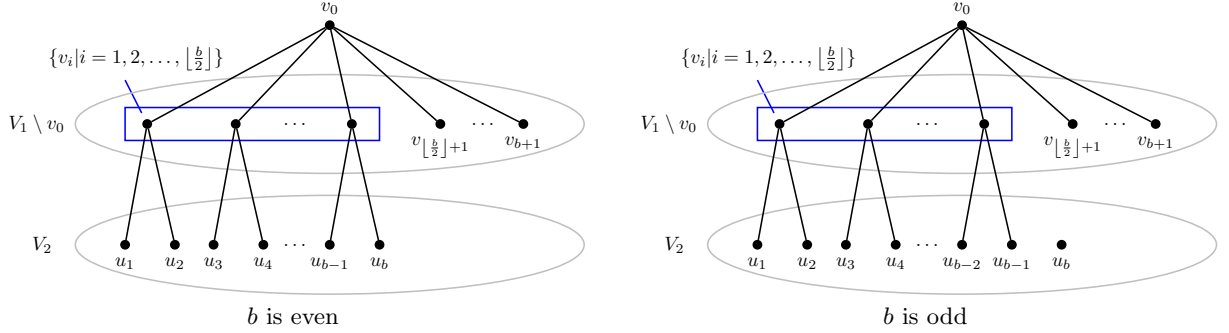
\begin{figure}[H]
\centering
\begin{minipage}[t]{0.48\textwidth}
  \centering
  \resizebox{\linewidth}{!}{%
    \begin{tikzpicture}[
        node/.style={circle, draw, fill=black, inner sep=1.5pt}, 
        dashed-node/.style={circle, dash pattern=on 0.8pt off 0.4pt, draw, inner sep=1.5pt}, 
        line width=0.8pt,
        font=\small
    ]

        \node[node] (v0) at (0, 4.0) {};
        \node at (0, 4.3) {$v_0$}; 
        
        \node[node] (v1) at (-3.3, 2.2) {};
        \node[node] (v2) at (-1.7, 2.2) {};
        \node at (-0.6,2.2) {$\dots$};
        \node[node] (v3) at (0.4, 2.2) {};
       
        \node[node,label=below:$v_{\left\lfloor \frac{b}{2}\right\rfloor+1}$] (v4) at (2.0, 2.2) {};
        \node at (2.8,2.2) {$\dots$};
        \node[node,label=below:$v_{b+1}$] (v5) at (3.5, 2.2) {};
        
        \draw[blue, thick] (-3.7, 1.9) rectangle (0.9, 2.5); 
        \node at (-3.5, 3.4) {$\{v_i|i=1,2,\dots,\left\lfloor \frac{b}{2}\right\rfloor\}$};
        \draw[blue] (-3.7, 3.0) -- (-3.4, 2.4); 

        \draw[lightgray] (0, 2.2) ellipse (4.6cm and 0.9cm); 
        \draw[lightgray] (0, 0) ellipse (4.6cm and 0.9cm); 
        \node at (-5.3, 2.2) {$V_{1}\setminus v_{0}$}; 
        \node at (-5.2, 0) {$V_{2}$}; 
        
        \node[node,label=below: $u_1$] (u1) at (-3.7, 0) {};
        \node[node,label=below: $u_2$] (u2) at (-2.8, 0) {};
        \node[node,label=below: $u_3$] (u3) at (-2.1, 0) {};
        \node[node,label=below: $u_4$] (u4) at (-1.2, 0) {};
        \node at (-0.6,0) {$\dots$};
        \node[node,label=below: $u_{b-1}$] (u5) at (0.0, 0) {};
        \node[node,label=below: $u_{b}$] (u6) at (0.9, 0) {};
        
        \draw (v0) -- (v1);
        \draw (v0) -- (v2);
        \draw (v0) -- (v3);
        \draw (v0) -- (v4);
        \draw (v0) -- (v5);

        \draw (v1) -- (u1); 
        \draw (v1) -- (u2);
        \draw (v2) -- (u3);
        \draw (v2) -- (u4);
        \draw (v3) -- (u5); 
        \draw (v3) -- (u6);       
\end{tikzpicture}}
\vspace{5pt}
\footnotesize{$b$ is even} 
\end{minipage}
\hfill
\begin{minipage}[t]{0.48\textwidth}
\centering
\resizebox{\linewidth}{!}{%
\begin{tikzpicture}[
node/.style={circle, draw, fill=black, inner sep=1.5pt}, 
        dashed-node/.style={circle, dash pattern=on 0.8pt off 0.4pt, draw, inner sep=1.5pt}, 
        line width=0.8pt,
        font=\small
    ]

        \node[node] (v0) at (0, 4.0) {};
        \node at (0, 4.3) {$v_0$}; 
        
        \node[node] (v1) at (-3.3, 2.2) {};
        \node[node] (v2) at (-1.7, 2.2) {};
        \node at (-0.6,2.2) {$\dots$};
        \node[node] (v3) at (0.4, 2.2) {};
       
        \node[node,label=below:$v_{\left\lfloor \frac{b}{2}\right\rfloor+1}$] (v4) at (2.0, 2.2) {};
        \node at (2.8,2.2) {$\dots$};
        \node[node,label=below:$v_{b+1}$] (v5) at (3.5, 2.2) {};
        
        \draw[blue, thick] (-3.7, 1.9) rectangle (0.9, 2.5); 
        \node at (-3.5, 3.4) {$\{v_i|i=1,2,\dots,\left\lfloor \frac{b}{2}\right\rfloor\}$};
        \draw[blue] (-3.7, 3.0) -- (-3.4, 2.4); 

        \draw[lightgray] (0, 2.2) ellipse (4.6cm and 0.9cm); 
        \draw[lightgray] (0, 0) ellipse (4.6cm and 0.9cm); 
        \node at (-5.3, 2.2) {$V_{1}\setminus v_{0}$}; 
        \node at (-5.2, 0) {$V_{2}$}; 
        
        \node[node,label=below: $u_1$] (u1) at (-3.7, 0) {};
        \node[node,label=below: $u_2$] (u2) at (-2.8, 0) {};
        \node[node,label=below: $u_3$] (u3) at (-2.1, 0) {};
        \node[node,label=below: $u_4$] (u4) at (-1.2, 0) {};
        \node at (-0.6,0) {$\dots$};
        \node[node,label=below: $u_{b-2}$] (u5) at (0.0, 0) {};
        \node[node,label=below: $u_{b-1}$] (u6) at (0.9, 0) {};
        \node[node,label=below:$u_{b}$] (u7) at (1.8, 0) {};
        
        \draw (v0) -- (v1);
        \draw (v0) -- (v2);
        \draw (v0) -- (v3);
        \draw (v0) -- (v4);
        \draw (v0) -- (v5);

        \draw (v1) -- (u1); 
        \draw (v1) -- (u2);
        \draw (v2) -- (u3);
        \draw (v2) -- (u4);
        \draw (v3) -- (u5); 
        \draw (v3) -- (u6); 
\end{tikzpicture}}
\vspace{5pt}
\footnotesize{$b$ is odd} 
\end{minipage}
\caption{Two cases of $H_{3}$}
\label{fig:H3}
\end{figure}

\hfill

If $s\le \left\lfloor\frac{b}{2}\right\rfloor$, then by \eqref{eq:q=b-2}, we obtain $e(G)\le e(H_{3})$. If $s\ge \left\lfloor\frac{b}{2}\right\rfloor+1$, then similar to the proof of Claim~\ref{clm:n=2b+1}, there are at most $\left\lceil\frac{b}{2}\right\rceil+1$ vertices in $V_{1}$ with degree $b+1$ and each of the remaining vertices in $V(G)$ has degree at most $b$, so we have $e(G)\le e(H_{3})$. Therefore, we can conclude that $\Delta(G)=b+1$ and $e(G)=\left\lfloor\frac{b(2b+2)+2+\left\lfloor\frac{b}{2}\right\rfloor}{2}\right\rfloor$. This completes the proof of Claim~\ref{clm:n=2b+2}.
\end{proof}

\begin{claim}\label{clm:n<2(b+4)} 
$n<2(b+4)$.
\end{claim}

\begin{proof}
By Lemma~\ref{lem:maxdegree}, we know $\Delta(G) \le b+2$ and then $e(G)\le \left\lfloor\frac{n(b+2)}{2}\right\rfloor$. Let $n=p(b+4)+q$, where $q\in \{0,1,...,b+3\}$. Let $\hat{G}\in \Ex(b+4+q, K_{1,b+1})$, by Lemma~\ref{lem:K1b}, we have $e(\hat{G})=\left\lfloor\frac{b(b+4+q)}{2}\right\rfloor$. If $(p-3)(b+4)-2q\ge 2$, then
\begin{align*}
e\left((p-1) K_{b+4} \cup \hat{G}\right) 
    &= (p-1) \frac{(b+4)(b+3)}{2} + \left\lfloor\frac{b(b+4+q)}{2}\right\rfloor \\
    &= \frac{(n-q-(b+4))(b+3)}{2} + \left\lfloor\frac{b(b+4+q)}{2}\right\rfloor \\
    &= \left\lfloor\frac{n(b+3)-q(b+3)-(b+4)(b+3)+b(b+4)+qb}{2}\right\rfloor \\
    &= \left\lfloor\frac{n(b+2)}{2} + \frac{n-3q-3(b+4)}{2}\right\rfloor = \left\lfloor\frac{n(b+2)}{2} + \frac{(p-3)(b+4)-2q}{2}\right\rfloor \\
    &> \left\lfloor\frac{n(b+2)}{2}\right\rfloor \geq e(G).
\end{align*}
Note that $(p-1) K_{b+4}\cup \hat{G}$ is $S_{3,b}$-free and $e((p-1) K_{b+4}\cup \hat{G})>e(G)$, a contradiction. So we have $(p-3)(b+4)-2q\le 1$, that is $p\le 4$. Then to prove $n<2(b+4)$, it suffices to show that $p\ne 2,3,4$. If $\Delta (G)\le b+1$, then we obtain $e(G)\le \left\lfloor\frac{n(b+1)}{2}\right\rfloor$ and $e((p-1) K_{b+4}\cup \hat{G} )=\left\lfloor\frac{n(b+1)}{2}+\frac{(2p-3)(b+4)-q}{2}\right\rfloor$. Note that when $p\ge 2$, we have 
\[
  e(G)< \max \left\{e((p-1) K_{b+4}\cup \hat{G}),e(p K_{b+4}\cup K_{q})\right\},
\]
a contradiction, so we get $p=1$ as desired. Therefore, we only need to consider $\Delta(G)=b+2$ and $p=2$ or $3$ or $4$.

\vspace{0.3cm}

\noindent\textbf{Case 1.} $p = 2$.

Note that if $0 \leq q \leq 2$  and  $b+1 \leq q \leq b+3$, then
\[
e(G) \leq \frac{n(b+2)}{2} = \frac{2(b+4)(b+2) + q(b+2)}{2} < e(2K_{b+4} \cup K_{q}),
\]
a contradiction. So we have $3\le q\le b$. Note that $\Delta(G)=b+2$ and $|V_{2}|=2(b+4)+q-(b+3)=b+5+q\le 2b+5$. Let $G_{1}\in \Ex(b+5+q, S_{3,b})$, we have 
\[
  e(G[V_{2}])\le \ex(b+5+q, S_{3,b})=e(G_{1}).
\]    
We claim that
\begin{equation}
  e(G_{1})\le \max \left\{ \frac{b(b+5+q)+2+\left\lfloor\frac{b}{2}\right\rfloor}{2}, \frac{(b+4)(b+3)+q(q+1)}{2} \right\}. \label{eq:p=2}
\end{equation}
If $G_{1}$ is connected, then by Claims~\ref{clm:nneq},~\ref{clm:n=2b+1} and~\ref{clm:n=2b+2}, we have $e(G_{1})\le \frac{b(b+5+q)+2+\left\lfloor\frac{b}{2}\right\rfloor}{2}$. Suppose that $G_{1}$ is not connected, then by Lemma~\ref{lem:n1n2}, we know $G_{1}$ has two connected components. Let $G_{1}=G_{1}^{1}\cup G_{1}^{2}$ with $|V(G_{1}^{1})|=n_{1}\ge n_{2}=|V(G_{1}^{2})|$. Note that $n_{1}\ge b+4$ and $n_{2}\le q+1\le b$, next we consider the following two cases:

\noindent (i) $n_1 \geq b+5$.

Note that $n_1<2(b+4)$ and $n_2 \leq q \leq b$, then again by Claims~\ref{clm:nneq},~\ref{clm:n=2b+1} and~\ref{clm:n=2b+2}, we have 
$e(G_{1}^{1})\le \frac{bn_1+2+\left\lfloor\frac{b}{2}\right\rfloor}{2}$
and $e(G_{1}^{2}) = \frac{n_{2} (n_{2}-1)}{2}$. So we have
\[
e(G_{1}) = e(G_{1}^{1}) + e(G_{1}^{2}) \leq \frac{bn_{1}+2+\left\lfloor\frac{b}{2}\right\rfloor}{2} + \frac{n_{2} (n_{2}-1)}{2} \leq \frac{b(b+5+q)+2+\left\lfloor\frac{b}{2}\right\rfloor}{2}.
\]

\noindent(ii) $n_{1} = b+4$.

Note that $n_{2}=q+1$ and then we have $e(G_{1}) = \frac{(b+4)(b+3)+q(q+1)}{2}$.
So \eqref{eq:p=2} is true. Then we obtain $e(G[V_{2}])\leq e(G_{1})\leq\max\left\{\frac{b(b+5+q)+2+\left\lfloor\frac{b}{2}\right\rfloor}{2},\frac{(b+4)(b+3)+q(q+1)}{2}\right\}$. Moreover, we know from \eqref{eq:b+2} that $e(G[V_{1}])+e(V_{1}, V_{2})\leq \frac{(b+3)(b+2)}{2}$. Therefore,
\begin{align*}
&e(G)=e(G[V_{1}])+e(V_{1}, V_{2})+e(G[V_{2}])\\
&\leq \frac{(b+3)(b+2)}{2} + \max \left\{
    \frac{b(b+5+q)+2+\left\lfloor \frac{b}{2}\right\rfloor}{2}, 
    \frac{(b+4)(b+3)+q(q+1)}{2}
    \right\} \\
&= \frac{(b+4)(b+3)}{2} + \max \left\{
    \frac{b(b+4+q) -(\left\lceil \frac{b}{2} \right\rceil + 4)}{2}, 
    \frac{(b+4)(b+3) + q(q-1) - (2b-2q+6)}{2}
    \right\} \\
&< \frac{(b+4)(b+3)}{2} + \max \left\{
    \left\lfloor\frac{b(b+4+q)}{2}\right\rfloor, 
    \frac{(b+4)(b+3)+q(q-1)}{2}
    \right\} \\
&= \max \left\{ 
    e(K_{b+4} \cup \hat{G}), 
    e(2K_{b+4} \cup K_{q}) 
    \right\},
\end{align*}
a contradiction. Thus we have $p\ne 2$.

\vspace{0.3cm}

\noindent\textbf{Case 2.} $p=3$.

Note that if $0\le q\le 3$ and $b\le q\le b+3$, then 
\[
e(G)\le \frac{n(b+2)}{2}=\frac{3(b+4)(b+2)+q(b+2)}{2}<e(3 K_{b+4}\cup K_{q}),
\]
a contradiction. So we have $4\le q\le b-1$. Note that $\Delta(G)=b+2$ and $|V_{2}|=3(b+4)+q-(b+3)=2b+9+q\le 3b+8$. Let $G_{2}\in \Ex(2b+9+q, S_{3,b})$, we have 
\[
  e(G[V_{2}])\le \ex(2b+9+q, S_{3,b})=e(G_{2}).
\]
We claim that
\begin{equation}
  e(G_{2})\le \max \left\{\frac{(b+4)(b+3)+b(b+5+q)+2+\left\lfloor\frac{b}{2}\right\rfloor}{2}, \frac{2(b+4)(b+3)+q(q+1)}{2} \right\}.
  \label{eq:p=3}
\end{equation}
Note that $G_{2}$ is not connected since $p\ne 2$. Moreover, by Lemma~\ref{lem:n1n2}, we know $G_{2}$ has two or three connected components. Following an analysis similar to the case $p=2$, we can verify that \eqref{eq:p=3} is true. Therefore, we have
\[
\begin{aligned}
&e(G)=e(G[V_{1}])+e(V_{1}, V_{2})+e(G[V_{2}])\\
    &\leq \frac{(b+3)(b+2)}{2} + \max \left\{\frac{(b+4)(b+3)+b(b+5+q)+2+\left\lfloor\frac{b}{2}\right\rfloor}{2}, \frac{2(b+4)(b+3)+q(q+1)}{2}
    \right\} \\
&= 2\cdot\frac{(b+4)(b+3)}{2} + \max \left\{
    \frac{b(b+4+q)-(\left\lceil \frac{b}{2} \right\rceil+4)}{2},
    \frac{(b+4)(b+3)+q(q-1)-(2b-2q+6)}{2}
    \right\} \\
&< 2\cdot\frac{(b+4)(b+3)}{2} + \max \left\{
    \left\lfloor\frac{b(b+4+q)}{2}\right\rfloor,
    \frac{(b+4)(b+3)+q(q-1)}{2}
    \right\} \\
&= \max \left\{
    e(2K_{b+4} \cup \hat{G} ),
    e(3K_{b+4} \cup K_q)
    \right\},
\end{aligned}
\]
a contradiction. So we know $p\ne 3$.

\vspace{0.3cm}

\noindent\textbf{Case 3.} $p=4$.

Similar to the cases $p=2,3$, we can prove $p=4$ is also impossible, so we have $p=1$ as desired. Here we omit the proof since the proof idea is the same as cases $p=2,3$. This proves Lemma~\ref{lem:connected}.
\end{proof}

\section{Proof of Theorem~\ref{thm:main}}\label{sec:main}
We are now ready to prove Theorem~\ref{thm:main}.

Let $G\in \Ex(n, S_{3,b})$ and let $G_{1},...,G_{t}$ be all the components of $G$ with $|V(G_{i})|=n_{i}$. Let $n_{1}\le n_{2}\le \cdots \le n_{t}$. Note that $G_{i}\in \Ex(n_{i}, S_{3,b})$. If $t=1$, then $G=G_{1}$ is connected. So by Lemma~\ref{lem:connected}, we have 
\begin{equation}
e(G) = 
\begin{cases}  
\left\lfloor \dfrac{bn+3}{2} \right\rfloor, & \textrm{ if } b\ge 11 \textrm{ and } n=2b+1; \\[10pt]
\left\lfloor \dfrac{bn+2+\left\lfloor \frac{b}{2} \right\rfloor}{2} \right\rfloor, & \textrm{ if } b\ge 11 \textrm{ and } n=2b+2; \\[10pt]
\left\lfloor \dfrac{bn}{2} \right\rfloor,  &\textrm{ otherwise}. \label{eq:t=1}
\end{cases} 
\end{equation}
Next we consider $t\ge 2$.

\begin{claim}\label{clm:ni<b+4}
$n_{i}\le b+4$ for $i\in \{1,2,...,t-1\}$.
\end{claim}

\begin{proof}
Suppose to the contrary that $n_{t-1}\ge b+5$, then $n_{t}\ge n_{t-1}\ge b+5$. So by Lemma~\ref{lem:connected}, we have
\[
e(G_{t-1}\cup G_{t})\leq \frac{bn_{t-1}+2+\left\lfloor\frac{b}{2}\right\rfloor}{2}+\frac{bn_{t}+2+\left\lfloor\frac{b}{2}\right\rfloor}{2} \leq \frac{b(n_{t-1}+n_{t})+b+4}{2}.
\]
Note that $n_{t-1}+n_{t}-(b+4)\ge b+6$. 
Let $H_{1}\in \Ex(n_{t-1}+n_{t}-(b+4), K_{1,b+1})$. By Lemma~\ref{lem:K1b}, we have $e(H_{1})=\left\lfloor\frac{b(n_{t-1}+n_{t}-(b+4))}{2}\right\rfloor$. Then we obtain
  \begin{align*}
  e(G_{t-1}\cup G_{t}) \leq \frac{b(n_{t-1}+n_{t})+b+4}{2}  &< \frac{(b+4)(b+3)}{2} + \left\lfloor\frac{b(n_{t-1}+n_{t}-(b+4))}{2}\right\rfloor \\
  &= e(K_{b+4}\cup H_{1}).
\end{align*}
Clearly $K_{b+4}\cup H_{1}$ is $S_{3,b}$-free, a contradiction. This completes the proof of Claim~\ref{clm:ni<b+4}.
\end{proof}

\begin{claim}\label{clm:ni>b+4}
$n_{i}\ge b+4$ for $i\in \{2,...,t\}$.
\end{claim}

\begin{proof}
Suppose to the contrary that $n_{2}\le b+3$, then $n_{1}\le n_{2}\le b+3$. So we have $G_{1}\cong K_{n_{1}}$ and $G_{2}\cong K_{n_{2}}$.

If $n_{1}+n_{2}\le b+4$, then by Lemma~\ref{lem:n1n2}, we have
\[
e(G_{1} \cup G_{2}) = \binom{n_{1}}{2} + \binom{n_{2}}{2} < \binom{n_{1} + n_{2}}{2} = e(K_{n_{1}+n_{2}}).
\]
Note that $K_{n_{1}+n_{2}}$ is $S_{3,b}$-free and $G_{1}\cup G_{2}\in \Ex(n_{1}+n_{2}, S_{3,b})$, a contradiction. 
If $n_{1}+n_{2}\ge b+5$, then by Lemma~\ref{lem:n1n2}, we have
\[
e(G_{1} \cup G_{2}) = \binom{n_{1}}{2} + \binom{n_{2}}{2} < \binom{b+4}{2} + \binom{n_{1}+n_{2}-(b+4)}{2} = e(K_{b+4} \cup K_{n_{1}+n_{2}-(b+4)}).
\]
Note that $n_{1}+n_{2}-(b+4)\leq b+2$, so we know $K_{n_{1}+n_{2}-(b+4)}$ is $S_{3,b}$-free, a contradiction. This proves Claim~\ref{clm:ni>b+4}.  
\end{proof}

According to Claims~\ref{clm:ni<b+4},~\ref{clm:ni>b+4}, we get the following results:
\[
n_{1} \leq b+4,\quad n_{2}=n_{3}= \cdots =n_{t-1}= b+4,\quad n_{t} \geq b+4.
\]
Since $G$ is an extremal graph, we obtain
\[
G_{1} \cong K_{n_{1}},\quad G_{2} \cong K_{b+4},\quad \cdots,\quad \ G_{t-1} \cong K_{b+4}.
\]
Next we consider $n_t$.

\vspace{0.3cm}

\noindent\textbf{Case 1.} $n_{t}=b+4$.

In this case, we know $G_t \cong K_{b+4}$. So we obtain $t-1 = p$ and $n_{1} = q$. Therefore,
\[
G \cong (t-1)K_{b+4} \cup K_{n_{1}} \cong pK_{b+4} \cup K_{q},
\]
and
\begin{equation}
e(G) = p\binom{b+4}{2} + \binom{q}{2} = \frac{p(b+4)(b+3) + q(q-1)}{2}. \label{eq:nt=b+4}
\end{equation}

\vspace{0.3cm}

\noindent\textbf{Case 2.} $n_{t}\ge b+5$.

By Lemma~\ref{lem:connected}, we know $e(G_{t})\le \frac{b n_{t}+2+\left\lfloor\frac{b}{2}\right\rfloor}{2}$. Next we show $n_{1}=b+4$. Let $H_{2}\in \Ex(n_{1}+n_{t}, K_{1,b+1})$. By Lemma~\ref{lem:K1b}, we have $e(H_{2})=\left\lfloor\frac{b(n_{1} + n_{t})}{2}\right\rfloor$. If $n_{1}\le b$, then for $b\le 10$, we have
\[
\begin{aligned}
e(G_{1} \cup G_{t}) 
&\leq \frac{n_{1}(n_{1}-1)}{2} + \frac{bn_{t}}{2} \leq \frac{b(n_{1}-1) + bn_{t}}{2} \\
&< \left\lfloor\frac{b(n_{1} + n_{t})}{2}\right\rfloor = e(H_{2}),
\end{aligned}
\]
a contradiction. For $b\ge 11$, we have
\[
\begin{aligned}
e(G_{1} \cup G_{t}) 
&\leq \frac{n_{1}(n_{1}-1)}{2} + \frac{bn_{t} + 2 + \left\lfloor\frac{b}{2}\right\rfloor}{2} \leq \frac{b(n_{1}-1) + bn_{t} + 2 + \left\lfloor\frac{b}{2}\right\rfloor}{2} \\
&< \left\lfloor\frac{b(n_{1} + n_{t})}{2}\right\rfloor = e(H_{2}),
\end{aligned}
\]
a contradiction. If $b+1\le n_{1}\le b+3$, then we have $n_{1}+n_{t}-(b+4)\geq b+2$. Let $H_{3}\in \Ex(n_{t-1}+n_{t}-(b+4), K_{1,b+1})$. By Lemma~\ref{lem:K1b}, we have $e(H_{3})=\left\lfloor\frac{b(n_{t-1}+n_{t}-(b+4))}{2}\right\rfloor$. Note that $n_{1}(n_{1}-(b+1))\le 2(b+3)$, so we obtain
\[
\begin{aligned}
  e(G_{1} \cup G_{t}) 
  &\leq \frac{n_{1}(n_{1}-1)}{2} + \frac{bn_{t} + 2 + \left\lfloor\frac{b}{2}\right\rfloor}{2} \\
  &= \frac{n_{1}(n_{1}-(b+1)) + b(n_{1}+n_{t}) + 2 + \left\lfloor\frac{b}{2}\right\rfloor}{2} \leq \frac{2(b+3) + b(n_{1}+n_{t}) + 2 + \left\lfloor\frac{b}{2}\right\rfloor}{2} \\
  &= \frac{(b+4)(b+3)}{2} + \frac{b(n_{1}+n_{t}-(b+4)) + 2 + \left\lfloor\frac{b}{2}\right\rfloor - (b+6)}{2} \\
  &< \binom{b+4}{2} + \left\lfloor\frac{b(n_{1}+n_{t}-(b+4))}{2}\right\rfloor = e(K_{b+4} \cup H_{3}),
\end{aligned}
\]
a contradiction. So we have $n_{1}=b+4$. Then we obtain $n_{1}=n_{2}=\cdots =n_{t-1}=b+4$. Note that $n_{t}\geq b+5$ and $G_{t}\in \Ex(n_{t}, S_{3, b})$ is connected, then by Lemma~\ref{lem:connected}, we know $n_{t}<2(b+4)$. Note that $n=(p-1)(b+4)+b+4+q$ and $q\ge 1$. So we have $t=p$, $n_{t}=b+4+q$ and then
\[
e(G) = e((p-1) K_{b+4} \cup G_{t}) = (p-1) \binom {b+4} {2} + e(G_{t}),
\]
where 
\[ 
e(G_{t}) = 
\begin{cases}  
\left\lfloor \dfrac{bn_{t} + 3}{2} \right\rfloor, & \textrm{ if } b\ge 11 \textrm{ and } n_{t} = 2b+1; \\[10pt]
\left\lfloor \dfrac{bn_{t}+2+\left\lfloor \frac{b}{2} \right\rfloor}{2} \right\rfloor, & \textrm{ if } b\ge 11 \textrm{ and } n_{t} = 2b + 2; \\[10pt]
\left\lfloor \dfrac{bn_{t}}{2} \right\rfloor,  &\textrm{ otherwise}. 
\end{cases}
\]
Together with \eqref{eq:t=1} and \eqref{eq:nt=b+4}, we obtain
\[
e(G) = \max\left\{(p-1)\binom {b+4}{2} + e(G_{t}),\ p\binom {b+4}{2}+ \binom{q}{2}\right\}.
\]
This completes the proof of Theorem \ref{thm:main}.
\hfill \qed

\section{Concluding remarks}\label{sec:conclu}
Let $a,b\in \mathbb{N}$ with $b>a$. Let $n=p(a+b+1)+q$, where $p,q\in \mathbb{N}$ and $0\le q\le a+b$. We have seen from Theorems~\ref{thm:a=1},~\ref{thm:a=2} and~\ref{thm:main} that for $a\le 3$, $\ex(n,S_{a,b})=p \binom{a+b+1}{2} + \binom{q}{2}$ when $q$ is close to $0$ or $a+b$. Now we show this behavior holds for larger $a$ as well.

\begin{theorem}\label{thm:general} Let $n,a,b\in \mathbb{N}$ with $b>a$ and $n\ge a+b+1$. Let $n=p(a+b+1)+q$, where $p,q\in \mathbb{N}$ and $0\le q\le a+b$. Then for $q=0,1,a+b$, we have $\ex(n,S_{a,b})=p \binom{a+b+1}{2} + \binom{q}{2}$ and $pK_{a+b+1}\cup K_{q}$ is the unique extremal graph.
\end{theorem}

\begin{proof}
Let $G\in \Ex(n, S_{a,b})$. Note that $pK_{a+b+1}\cup K_{q}$ is $S_{a,b}$-free, then we have $e(G)\ge e(pK_{a+b+1}\cup K_{q})$. By Lemma~\ref{lem:maxdegree}, we know $\Delta(G)\le a+b$. If $q=0$, then we have
\[
e(G)\le\frac{p(a+b+1)(a+b)}{2}=e(pK_{a+b+1}).
\]
So we get $e(G)=p \binom{a+b+1}{2}$ and $pK_{a+b+1}$ is the unique extremal graph. Now we consider $q\in \{1, a+b\}$. Let $G_{1},...,G_{t}$ be all the components of $G$ with $|V(G_{i})|=n_{i}$, note that $G_{i}\in \Ex(n_{i}, S_{a,b})$. Next we prove that $n_{i}\le a+b+1$ for all $i$. Assume that there exists $j\in \{1,...,t\}$ such that $n_{j}\ge a+b+2$, then by Lemma~\ref{lem:maxdegree}, we have $\Delta (G_{j})\le a+b-1$. Therefore,
\[
\begin{aligned}
e(G)=\sum_{i=1}^t e(G_{i})&=\sum_{i\ne j} e(G_{i})+e(G_{j})\\
&\le \frac{(p(a+b+1)+q-n_{j})(a+b)+n_{j}(a+b-1)}{2}\\
&=\frac{p(a+b+1)(a+b)+q(a+b)-n_{j}}{2}\\
&\le \frac{p(a+b+1)(a+b)+(a+b)(q-1)-2}{2}\\
&< \frac{p(a+b+1)(a+b)+q(q-1)}{2}\\
&=e(pK_{a+b+1}\cup K_{q}),
\end{aligned}
\]
a contradiction. So we obtain $n_{i}\le a+b+1$ for each $i$ and then $G_{i}\cong K_{n_{i}}$. Moreover, by Lemma~\ref{lem:n1n2} and since $q\ne 0$, we know that there is exactly one component of $G$ with the number of vertices less than $a+b+1$. So we have $t=p+1$, then $G\cong pK_{a+b+1}\cup K_{q}$ and $e(G)=p \binom{a+b+1}{2}+ \binom{q}{2}$. This completes the proof of Theorem~\ref{thm:general}.
\end{proof}

Indeed, determining the function $\ex(n,S_{a,b})$ for general $a$ remains an open and challenging problem. Next we provide some insights on $\ex(n,S_{a,b})$.

\begin{remark}
Let $n,a,b\in \mathbb{N}$ with $b>a$ and $n\ge a+b+2$. Let $G\in \Ex(n,S_{a,b})$ and suppose that $G$ is connected, then  
\end{remark}

\begin{enumerate}[$(i)$]
    \item If $a=1$ or $2$, then $\Delta(G)=b$ and $G$ is a near $b$-regular graph.
    \item If $a=3$, then $\Delta(G)=b$ or $b+1$ and $G$ has more edges than the near $b$-regular graph when $n=2b+1$ or $2b+2$.
\end{enumerate}
Next we construct an $n$-vertex $S_{a,b}$-free graph $H$ to show that for larger $a$ and sufficiently large $b$, $G$ has more edges than the $b$-regular graph for a wider range of $n$.

Let $n=a+b+1+q$, where $b-2a+3\le q\le b-a+1$. Let $k=b-a-q$, then $-1\le k\le a-3$.
Let $s=\left\lceil\frac{(k+2)(b-k-1)}{a-1}\right\rceil$.
Suppose $v_{0}\in V(H)$ and $d_{H}(v_{0})=\Delta(H)=b+1$. Let $N_{H}(v_{0})=\{v_{1},v_{2},...,v_{b+1}\}$, $V_{1}=N_{H}[v_{0}]$, $V_{2}=V(H)\setminus V_{1}=\{u_{1},u_{2},...,u_{b-k-1}\}$ be the set of vertices such that $d(u_{i},v_{0})=2$, and $v_{1},v_{2},...,v_{s}$ be all the vertices in $N_{H}(v_{0})$ adjacent to some vertices in $V_2$.
As shown in Fig.~\ref{fig:H general} (for simplicity, we omit the internal structures of the three ellipses in the figure):
\begin{itemize}
  \item $d_{V_{2}}(v_{i})=a-1$ for each $i\in \{1,2,...,s-1\}$ and $d_{V_{2}}(v_{s})=(k+2)(b-k-1)-(s-1)(a-1)$;
  \item $d_{H}(v_{0}) = d_{H}(v_{s+1}) =\cdots= d_{H}(v_{b+1}) = b+1$;
  \item For each $i\in \{1,2,...,s-1\}$, the degree of $v_{i}$ in $H[\{v_{1},...,v_{s}\}]$ is $s-a-1$. And the degree of $v_{s}$ in $H[\{v_{1},...,v_{s}\}]$ is either $s-2-(k+2)(b-k-1)+(s-1)(a-1)$ or $s-3-(k+2)(b-k-1)+(s-1)(a-1)$;
  \item $d_{V_1}(u_{j})=k+2$ for each $j\in \{1,2,...,b-k-1\}$;
  \item $H[V_{2}]\cong K_{b-k-1}$.
\end{itemize}

\hfill

\begin{figure}[htbp]
\centering
    \begin{tikzpicture}[
        node/.style={circle, draw, fill=black, inner sep=1.5pt}, 
        dashed-node/.style={circle, dash pattern=on 0.8pt off 0.4pt, draw, inner sep=1.5pt}, 
        line width=0.8pt,
        font=\footnotesize
    ]

        \node[node] (v0) at (0, 4.0) {};
        \node at (0, 4.3) {$v_0$}; 
        
        \node[node] (v1) at (-3.4, 2.2) {};
        \node[node] (v2) at (-2.9, 2.2) {};
        \node at (-2.5,2.2) {$\dots$};
        \node[node] (v3) at (-2.2, 2.2) {};
        \node at (-1.7,2.2) {$\dots$};
        \node[node] (v4) at (-1.3, 2.2) {};
        \node[node] (v5) at (-0.8, 2.2) {};
        \node at (-0.4,2.2) {$\dots$};
        \node[node] (v6) at (-0.1, 2.2) {};
        \node at (0.4,2.2) {$\dots$};
        \node[node] (v7) at (0.9, 2.2) {};
        \node at (1.4,2.2) {$\dots$};
        \node[node] (v9) at (1.9, 2.2) {};
        \node[node] (v10) at (2.8, 2.2) {};
        \node at (3.3,2.2) {$\dots$};
        \node[node] (v11) at (3.8, 2.2) {};
        
        \draw[blue, thick] (-3.7, 1.9) rectangle (2.2, 2.5); 
        \node at (-4.2, 3.7) {$\{v_i|i=1,2,\dots,k+2,\dots,\left\lceil\frac{(k+2)(b-k-1)}{a-1}\right\rceil\}$};
        \draw[blue] (-3.7, 3.1) -- (-3.4, 2.4); 

        \draw[blue, thick] (2.3, 1.9) rectangle (4.0, 2.5); 
        \node at (4.0, 3.7) {$\{v_i|i=\left\lceil\frac{(k+2)(b-k-1)}{a-1}\right\rceil+1,\dots,b+1\}$};
        \draw[blue] (3.6, 2.4) -- (3.9, 3.1); 

        \draw[blue, thick] (-3.7, -0.3) rectangle (2.4, 0.3); 
        \node at (-3.0, -1.3) {$\{u_j|j=1,2,\dots,a-1,\dots,b-k-1\}$};
        \draw[blue] (-3.8, -0.8) -- (-3.3, -0.2); 

        \draw[lightgray] (1.4, 1.1) ellipse (1.1 and 1.5); 
        
        \draw[lightgray] (0, 2.2) ellipse (4.6cm and 0.9cm); 
        \draw[lightgray] (0, 0) ellipse (4.6cm and 0.9cm); 
        \node at (-5.3, 2.2) {$V_{1}\setminus v_{0}$}; 
        \node at (-5.2, 0) {$V_{2}$}; 
        
        \node[node] (u1) at (-3.4, 0) {};
        \node[node] (u2) at (-2.9, 0) {};
        \node at (-2.5,0) {$\dots$};
        \node[node] (u3) at (-2.2, 0) {};
        \node at (-1.7,0) {$\dots$};
        \node[node] (u4) at (-1.3, 0) {};
        \node[node] (u5) at (-0.8, 0) {};
        \node at (-0.4,0) {$\dots$};
        \node[node] (u6) at (-0.1, 0) {};
        \node at (0.4,0) {$\dots$};
        \node[node] (u7) at (0.9, 0) {};
        \node at (1.4,0) {$\dots$};
        \node[node] (u7) at (1.9, 0) {};
        
        \draw (v0) -- (v1);
        \draw (v0) -- (v2);
        \draw (v0) -- (v3);
        \draw (v0) -- (v4);
        \draw (v0) -- (v5);
        \draw (v0) -- (v6);
        \draw (v0) -- (v7);
        \draw (v0) -- (v9);
        \draw (v0) -- (v10);
        \draw (v0) -- (v11);

        \draw (v1) -- (u1); 
        \draw (v1) -- (u2);
        \draw (v1) -- (u3); 
        \draw (v2) -- (u1);
        \draw (v2) -- (u2); 
        \draw (v2) -- (u3);
        \draw (v3) -- (u1);
        \draw (v3) -- (u2); 
        \draw (v3) -- (u3); 
        
        \draw (v4) -- (u4); 
        \draw (v4) -- (u5);
        \draw (v4) -- (u6); 
        \draw (v5) -- (u4);
        \draw (v5) -- (u5); 
        \draw (v5) -- (u6);
        \draw (v6) -- (u4);
        \draw (v6) -- (u5); 
        \draw (v6) -- (u6); 
        
\end{tikzpicture}
\caption{$H$}
\label{fig:H general}
\end{figure}
\hfill
Since $\Delta (H)=b+1$ and only the vertices $v_{0},v_{s+1},...,v_{b+1}$ attain degree $b+1$, it is easy to see that $H$ is $S_{a,b}$-free. Moreover, the remaining vertices in $V(H)$ have degree $b$, except for $v_{s}$, whose degree is either $b$ or $b-1$. So we know $H$ has more edges than the near $b$-regular graph on $a+b+1+q$ vertices. Note that
\[
\begin{aligned}
e(H)=\frac{1}{2}\sum_{v\in V_1}d_{H}(v)+\frac{1}{2}\sum_{v\in V_2}d_{H}(v)&\ge \frac{b(a+b+1+q)+b-s+1}{2}\\
&\ge \frac{b(a+b+1+q)+b-\frac{(b-a-q+2)(a-1+q)}{a-1}}{2}.   
\end{aligned}
\]
Then we have
\[
\begin{aligned}
&e(H)-e(K_{a+b+1}\cup K_{q})\\
&\ge \frac{b(a+b+1+q)+b-\frac{(b-a-q+2)(a-1+q)}{a-1}-(a+b+1)(a+b)-q(q-1)}{2}\\
&= \frac{(a-2)q(b-q+2)-(a-1)a(a+b)+aq-2(a-1)}{2(a-1)}.
\end{aligned}
\]
Let $f(q)=(a-2)q(b-q+2)-(a-1)a(a+b)+aq-2(a-1)$. Since $a\ge 3$, we know the leading coefficient of $f(q)$ is negative. Moreover, since $a\ge 3$ and $b$ is sufficiently large, we obtain $f(b-2a+3)>0$ and $f(b-a+1)>0$. So we have $f(q)>0$ for $b-2a+3\le q\le b-a+1$, then $e(H)>e(K_{a+b+1}\cup K_{q})$.

Hence for any $a\ge 3$, there exists $a-1$ particular cases where $G$ has more edges than $K_{a+b+1}\cup K_{q}$ or the near $b$-regular graph on $a+b+1+q$ vertices. This also explains why in Theorem~\ref{thm:main} ($a=3$), we have two more cases than Theorem~\ref{thm:a=1} and \ref{thm:a=2}. In fact, for $a=2$, a special case with $q=b-1$ also exists. However, the graph in this case has fewer edges than $K_{a+b+1}\cup K_{b-1}$, so it does not exist in Theorem~\ref{thm:a=2}.

This pattern implies that larger $a$ leads to more cases where the connected extremal graph outperforms the near $b$-regular graph. It also hints that for the function $\ex(n,S_{a,b})$, a larger $a$ corresponds to a wider $n$-range with extremal values deviating from previous results, and the extremal graph may become more diverse.


\bigskip

\bibliographystyle{plain} 
\bibliography{references}

\end{document}